\documentclass[12pt, oneside]{amsart}       % use "amsart" instead of "article" for AMSLaTeX format
\usepackage{geometry}    
\usepackage{amsmath}  
\usepackage{amssymb}                  % See geometry.pdf to learn the layout options. There are lots.
\usepackage{graphicx}                % Use pdf, png, jpg, or epsÂ§ with pdflatex; use eps in DVI mode
                            \usepackage{amsthm}    
\newtheorem{theorem}{Theorem}[section]
\newtheorem{corollary}{Corollary}[section]
\newtheorem{definition}{Definition}[section]
\newtheorem{lemma}[theorem]{Lemma}
\theoremstyle{remark}
\newtheorem*{remark}{Remark}
\usepackage{tcolorbox}
\newtheorem*{example}{Example}

\title{Appell-Lerch sums and $\mathcal{N}=2$ Moduli}
\author{Emile Bouaziz}

\begin{document}
\maketitle

\begin{abstract} We study moduli of odd-framed $\mathcal{N}=2$ elliptic curves subject to certain conditions, and show that the fermionic part of the moduli problem is essentially controlled by the Appell-Lerch sum, familiar from the theory of mock modular forms. \end{abstract}

\section{Introduction} The goal of this short note, which is inspired greatly by the works of Polischuk, \cite{Pol}, and Zwegers, \cite{Zw}, is to provide a (super-) moduli theoretic interpretation of the Appell-Lerch sum, which in our normalisation will be  $$\kappa(u,v;\tau)=\kappa(x,y;q):=\vartheta^{-1}(x)\sum_{n\in\mathbb{Z}}\frac{q^{\binom{n}{2}}(-x)^{n}}{q^{n}-y},$$ for $x=e^{2\pi iu}, y=e^{2\pi iv}, q=e^{2\pi i\tau}$, with $\mathfrak{Im}(\tau)>0$ and $\vartheta(x):=\sum_{n\in\mathbb{Z}}q^{\binom{n}{2}}(-x)^{n}$. 

The function $\kappa$ is by now well known to appear naturally in certain questions arising in Super-Conformal Field Theory, and our approach is in keeping with this as it centres on the study of a certain moduli space of \emph{super-curves}. By this we mean $1|1$ dimensional smooth and proper complex manifolds, which, in keeping with the physical literature, we will refer to as $\mathcal{N}=2$ curves. The appearance of $\kappa$ will be a purely fermionic phenomenon, invisible at the level of the underlying even spaces. In a sense, (the poles of) $\kappa$ will measure the \emph{non-splitness} of an appropriate universal family, which can be taken to be the slogan animating this text.

Slightly more precisely, we will restrict our attention to pointed super curves with fixed topological invariants corresponding to elliptic curves (the reduced part) with degree one line bundles (the odd fibre directions), equipped further with a certain decoration. The bulk of the geometry is thus very well understood and is easily described (for fixed $E$ say) in terms of the dual elliptic curve, $E^{\vee}$, and the (degree $1$ component of the) Poincar\'e sheaf. The interesting phenomena are thus purely odd, so that everything is nilpotent and we are really doing deformation theory. We shall see below that $\kappa$ arises as a meromorphic splitting of an extension of line bundles corresponding to the \emph{odd cotangent exact sequence} attached to the universal $\mathcal{N}=2$ curve with the aforementioned topological invariants.

We begin by studying the moduli problem for a fixed elliptic curve $E$, which essentially amounts to some rather basic function theory on the curve $E$. We will then explain how the results can be globalized so as to allow us to vary the curve. Along the way so we will give algebro-geometric interpretations of some of the constructions arising in the celebrated work of Zwegers, although such intepretations are surely well-known to experts. In particular we will explain the algebraic significance of his non-holomorphic correction function $R$. We will see in particular that it corresponds to an $\mathbb{R}$-analytic \emph{splitting} of the universal $\mathcal{N}=2$, as well as providing a representative in $\bar{\partial}$-cohomology for the extension of line bundles described above.

\section{Constructions with fixed elliptic curve}\subsection{Construction of Odd Deformations} For a detailed introduction to the theory of super-spaces we refer to the work \cite{Wi}. We work over $\mathbb{C}$, and we will pass between the algebraic and holomorphic worlds without further comment. A \emph{family of} $\mathcal{N}=2$ curves will mean a smooth morphism $\mathcal{X}\rightarrow B$ with fibres of dimension $1|1$. With $B=*$ we have that $\mathcal{X}=(X\,|\,L)$ is a \emph{split} space, given as the odd total space of a line bundle $L$ on a smooth curve $X$. We have $X=\mathcal{X}^{red}$, the \emph{reduced} part of $\mathcal{X}$, and $L=\Omega^{0|1}_{\mathcal{X}}$, the odd cotangent sheaf, which is a sheaf on the reduced part. A proper $\mathcal{N}=2$ curve, $\mathcal{X}$, has two topological invariants given by $g(X)$ and $c_{1}(L)$. We recall here that a map $\mathcal{Y}\rightarrow\mathcal{X}$ is called \emph{split} if it realizes $\mathcal{X}$ as the total space of an odd vector bundle on $\mathcal{Y}$. $\mathcal{X}$ is called split if the map $\mathcal{X}^{red}\rightarrow\mathcal{X}$ is split.

\begin{definition} A family $\pi:\mathcal{X}\rightarrow B$ is called \emph{split} if $\pi^{-1}(B^{red})\rightarrow\mathcal{X}$ is split relative to $B$. \end{definition} 
\begin{remark} If $B$ is purely even then any family of $\mathcal{N}=2$ curves over $B$ is automatically split. \end{remark}

Infinitessimal deformations of the $\mathcal{N}=2$ curve $(X\,|\,L)=\mathcal{X}$ are parametrized by $H^{1}(\mathcal{X},T_{\mathcal{X}})$. We can compute this via the projection $p:\mathcal{X}\rightarrow X$, noting that this map is affine and that $p_{*}T_{\mathcal{X}}\cong T_{X}\oplus\mathcal{O}_{X}\oplus\Pi (L^{-1}\oplus T_{X}\otimes L)$, where $\Pi$ denotes the parity-switching functor. Infinitessimal odd deformations are thus parameterized by $H^{1}(X,L^{-1})\oplus H^{1}(X,L\otimes T_{X})$. The topological restrictions we will later impose lead us to focus on the $H^{1}(X,L^{-1})$ summand. The following lemma, whilst trivial, is nonetheless central to our approach. \begin{lemma} An extension $\mathcal{O}\rightarrow V\rightarrow L$ induces a deformation of $(X\,|\,L)$ with total space $(X\,|\,V)$. This deformation is trivial iff it is split in the sense of def 2.1 which is further equivalent to the extension $V$ being split. This establishes a bijection between $H^{1}(X,L^{-1})$ and isomorphism classes of deformations of $\mathcal{X}$ over $\mathbb{A}^{0|1}$ with split total space. \end{lemma}

\begin{proof} A map $(X|V)\rightarrow \mathbb{A}^{0|1}$ induces a global section, $\xi\in\Gamma(X,V)$ by taking the pull-back of the natural global odd cotangent vector on $\mathbb{A}^{0|1}$. Smoothness of the map is equivalent to $\mathsf{coker}(\xi)$ being a vector bundle, and the fibre of $(X\,|\,V)\rightarrow \mathbb{A}^{0|1}$ over $*=\mathbb{A}^{0|1,red}$ is $(X\,|\,\mathsf{coker}(\xi))$. That there is an induced bijection as claimed can now be easily checked.\end{proof}

The following lemma will be of some use to us;

\begin{lemma} Let $f:\mathcal{X}\rightarrow\mathcal{Y}$ be a morphism of smooth super-spaces, so that $\mathcal{X}^{red}\rightarrow\mathcal{Y}^{red}$ is an isomorphism. Then $f$ is an isomorphism iff it induces an isomorphism on $\mathbb{A}^{\,0|1}$-valued points. This is moreover the case iff the induced map $(f^{red})^{*}\Omega_{\mathcal{Y}}^{\,0|1}\rightarrow\Omega^{\,0|1}_{\mathcal{X}}$ is an isomorphism.\end{lemma}

\begin{proof} This follows from the fact that a smooth super algebra $\mathcal{A}$, with $A:=\mathcal{A}^{red}$, admits a natural filtration with associated graded $Sym_{A}(\Pi\,\Omega_{\mathcal{A}}^{\,0|1})$. \end{proof}

\subsection{Elliptic Curves and Theta Bundles} Fix an elliptic curve $E=E_{q}=\mathbb{C}^{*}/q^{\mathbb{Z}}$, which we write in multiplicative notation. $E^{\vee}$ will denote the dual elliptic curve, which is isomorphic to $E$ via the Abel-Jacobi map. We will let $x$ denote a multiplicative coordinate on $E$, and $y$ a multiplicative coordinate on $E^{\vee}$. The bundle $\mathcal{O}_{E}(e)$, where $e\in E(\mathbb{C})$ is the neutral element, will be denoted $\Theta(x)$. $\Theta(y)$ will denote the corresponding bundle on $E^{\vee}$. Making use of the group structure we define also the bundle $\Theta(xy)$ on $E\times E^{\vee}$ in an evident manner. Descent with respect to the covering $\mathbb{C}^{*}\rightarrow E_{q}$ identifies isomorphism classes of line bundles on $E_{q}$ with classes in group cohomology $H^{1}(q^{\mathbb{Z}},\mathcal{O}_{\mathbb{C}^{*}}^{*})$, which are referred to as \emph{automorphy factors}, see subsection 3.1 below for a description of this formalism in a less trivial case. Global sections of the corresponding line bundle are given by functions with transformation properties (with respect to $q^{\mathbb{Z}}$) specified by the automorphy factor. $\Theta(x)$ corresponds to the automorphy factor $q^{n}\mapsto (-x)^{-n}q^{-\binom{n}{2}}$. There is a one-dimensional space of global sections of $\Theta(x)$ spanned by $\vartheta(x)=\vartheta(x;q)=\sum (-x)^{n}q^{\binom{n}{2}}$, the theta function corresponding to the \emph{odd} spin structure. A point $y\in E^{\vee}(\mathbb{C})$ specifies a line bundle $L_{y}$ on $E$, with $c_{1}(L_{y})=0$, with corresponding automorphy factor $q^{n}\mapsto y^{n}$. 

\begin{remark} We have $\Theta(xy)|_{E}\cong \Theta(x)\otimes L_{y}^{-1}$, as can be confirmed by comparing their respective automorphy factors. \end{remark}

We consider now moduli of (weakly) pointed $\mathcal{N}=2$ curves, $\mathcal{X}$, with $\mathcal{X}^{red}=E$, and $c_{1}(\Omega^{0|1})=1$ fibre-wise. By a \emph{weakly pointed family} we mean a family of $\mathcal{N}=2$ curves, equipped with a distinguished section of the induced map of reduced spaces. In fact we will consider a certain decorated moduli space so as to rigidify the situation somewhat. We will essentially consider families of curves equipped with non-zero odd cotangent vectors, up to isomorphisms. 

\begin{definition} Let $\mathcal{M}^{\mathcal{N}=2}_{E}$ denote the moduli functor parameterizing the following data up to isomorphism; \begin{itemize}\item a family, $\pi$, of $\mathcal{N}=2$ curves with reduced part $E$ and $c_{1}(\Omega^{0|1})=1$ fibre-wise, \item a distinguished section of the map $\pi^{red}$, \item a trivialization of the determinant of the push-forward of the odd cotangent sheaf $\mathbb{R}^{0}\pi_{*}\Omega^{0|1}_{\pi}$, which we refer to as an \emph{odd framing} of the family.\end{itemize} Further let $\mathcal{C}^{\mathcal{N}=2}_{E}\rightarrow\mathcal{M}^{\mathcal{N}=2}_{E}$ denote the resulting universal family. \end{definition}

\begin{remark} An $\mathcal{N}=2$ curve (over a point) always has a $\mathbb{C}^{*}$ of automorphisms given by re-scaling odd fibre directions. The data of the odd framing rigidifies the curve, so that we no longer have such automorphisms.  \end{remark}

If we restrict to even bases the moduli space is simply $E^{\vee}$ and we claim the universal curve is $$\pi:(E\times E^{\vee}\,|\,\Theta(xy))\rightarrow E^{\vee},$$ with the evident section. Certainly this will follow if we show $\mathbb{R}^{0}\pi_{*}\Theta(xy)\cong\mathcal{O}$, which we do in lemma 2.3 below. The fibre over $y\in E^{\vee}(\mathbb{C})$ is the $\mathcal{N}=2$ curve $(E\,|\,\Theta(x)L_{y}^{-1})$. 

\begin{remark} Note that we now see now another reason for the inclusion of the trivialization of $\mathbb{R}^{0}\pi_{\*}(\Omega^{0|1})$. Indeed otherwise we could equally well have taken $\Theta(xy)\otimes\Theta(y)$. \end{remark}

We will now calculate the genuine moduli space, which is necessarily a purely odd thickening of $E^{\vee}$, with universal curve an odd thickening of $(E\times E^{\vee}\,|\,\Theta(xy))$. We calculate first the sheaf on the base parameterizing deformations along the fibres. A relative version of the compupation of infinitessimal deformations from subsection 2.1 identifies this with $\mathbb{R}^{1}\pi_{*}\Theta(xy)^{-1}\cong\mathbb{R}^{0}\pi_{*}\Theta(xy)$. 

\begin{lemma} There is an isomorphism $\mathbb{R}^{0}\pi_{*}\Theta(xy)\cong\mathcal{O}$, whence an odd framing of the family above. \end{lemma}

\begin{proof} Let us first note that the fibres $H^{1}(E,\Theta L_{y}^{-1})$ are one dimensional by Riemann-Roch, so we have a line bundle. The space of global sections, $H^{0}(E\times E^{\vee},\Theta(xy))$, is one dimensional. 

The anti-diagonal on $E\times E^{\vee}$ has vanishing self-intersection, whence Riemann-Roch for the surface $E\times E^{\vee}$ implies that we have $\chi(\Theta(xy))=0$. We conclude that $H^{1}$ is one dimensional, as $H^{2}$ vanishes by Serre-Duality. 

Certainly $\mathbb{R}^{1}\pi_{*}\Theta(xy)=0$ and now considering the Leray spectral sequence for the projection, together with the computations of $H^{1}$ and $H^{2}$, we deduce that the line bundle $\mathbb{R}^{0}\pi_{*}\Theta(xy)$ has non-vanishing $H^{0}$ and $H^{1}$. It is thus necessarily $\mathcal{O}$. That $\mathbb{R}^{0}\pi_{*}$ is also trivial is now immediate from Grothendieck duality. \end{proof}

Note that the above implies that the sheaf on the base parametrizing odd deformations fibrewise is just $\mathcal{O}$. This suggests that we have $\mathcal{M}^{\mathcal{N}=2}_{E}\cong (E^{\vee}\,|\,\mathcal{O})$, so we should look for a universal family on this space. We saw in the proof above that there is a one dimensional space of extensions of $\mathcal{O}$ by $\Theta(xy)$. We will now describe the resulting sheaf more explicitly in automorphic terms, and in particular we will see that the extensions realizes fibrewise the extensions corresponding to the odd deformations of $(E\,|\,\Theta(x)L_{y}^{-1})$ described by lemma 2.1 above.

\begin{lemma} The unique non-trivial extension $\mathcal{O}\rightarrow\mathcal{K}\rightarrow\Theta(xy)$ is such that for all $y\in E^{\vee}(\mathbb{C})$, the corresponding extension class $[\mathcal{K}_{y}]$ is non-zero.\end{lemma} \begin{proof} We give a model for the automorphy factor representing $\mathcal{K}^{*}$. Consider the vector of meromorphic functions on $E\times E^{\vee}$, $$s(x,y):=\binom{\kappa(x,y)}{1},$$ where $$\kappa(x,y)=\vartheta^{-1}(x)\sum\frac{q^{\binom{n}{2}}(-x)^{n}}{q^{n}-y}$$ is the Appell-Lerch sum. It is elementary to check that the vector $s(x,y)$ transforms according to the automorphy factors  $$\binom{1}{0}\mapsto\begin{bmatrix}    

\,-xy&-x\\
\,0&1\\
\end{bmatrix}, \binom{0}{1}\mapsto\begin{bmatrix}\,-xy&-x\\ \,0 &1\end{bmatrix}.$$ Observe that this describes an extension of $\mathcal{O}$ by $\Theta^{-1}$, as the matrices are upper triangular with the appropriate diagonal entries. We let this define the sheaf $\mathcal{K}^{*}$. Now to check that for any $y$ the extension corresponding to the fibre $\mathcal{K}^{*}_{y}$ is non-split one must show that there is never a holomorphic lift of the section $1\in H^{0}(\mathcal{O})$.
 
To see this first consider $y\neq 1$ (modulo $q^{\mathbb{Z}}$). Multiplying by $\vartheta$ we would have a holomorphic function $s(x)$, which is divisible by $\vartheta$, such that $s(qx)-ys(x)=\vartheta$. There is however a unique such function, given by $\vartheta\kappa$, which is not (holomorphically) divisible by $\theta$. Indeed the sub-bundle $L_{y}$ of the corresponding extension has no sections for such $y$. Note that this argument appears in \cite{Pol}. Now let $y=1$ (modulo $q^{\mathbb{Z}}$). In this case we are looking for a holomorphic function $s$ so that $s(qx)-s(x)=\vartheta$, which is impossible by considering Laurent expansions around $0$. \end{proof}

\begin{tcolorbox}\begin{corollary} We have $\mathcal{M}^{\mathcal{N}=2}_{E}\cong (E^{\vee}\,|\,\mathcal{O})$ and $\mathcal{C}^{\mathcal{N}=2}_{E}\cong (E\times E^{\vee}\,|\,\mathcal{K})$. The family, $\pi:\mathcal{C}^{\mathcal{N}=2}_{E}\rightarrow\mathcal{M}^{\mathcal{N}=2}_{E}$ is induced from the extension according to lemma 2.1. The Appell-Lerch sum $\kappa(x,y)$ describes a meromorphic splitting of the family $\pi$.\end{corollary}\end{tcolorbox}

\begin{proof} We have the family, $$\pi:(E\times E^{\vee}\,|\,\mathcal{K})\longrightarrow (E^{\vee}\,|\,\mathcal{O}),$$  induced from the extension. The extension defining $\mathcal{K}$ is the short exact sequence for the odd cotangent spaces, whence the bundle $\Omega^{0|1}_{\pi}$ is isomorphic to $\Theta(xy)$, and the family is thus framed. By definition this produces a morphism $(E^{\vee}\,|\,\mathcal{O})\rightarrow\mathcal{M}^{\mathcal{N}=2}_{E}$, such that $(E\times E^{\vee}\,|\,\mathcal{K})$ is pulled back from $\mathcal{C}^{\mathcal{N}=2}_{E}$.

 It is an isomorphsim on reduced spaces, whence to see that is an isomorphism we must check that is an isomorphism on $\mathbb{A}^{\,0|1}$-valued points, by lemma 2.2. This is the content of lemma 2.4, which says precisely that the family, $(E\times E^{\vee}\,|\,\mathcal{K})\rightarrow (E^{\vee}\,|\,\mathcal{O}),$ sees all the $\mathbb{A}^{\,0|1}$-families of curves fibre-wise. 

That $\kappa$ describes a meromorphic splitting (cf definition 2.1 above) follows from the isomorphism $\pi^{-1}(\mathcal{M}^{\mathcal{N}=2,red}_{E})\cong (E\times E^{\vee}\,|\,\mathcal{O})$, and the fact that a splitting of an extension of $\mathcal{O}$ is a lift of the section $1\in H^{0}(\mathcal{O})$. \end{proof}

\section{Varying the curve}
\subsection{Automorphic and Geometric Data}
We collect here some general facts about sections of extensions of vector bundles expressed in automorphic language. In the next subsection we will specialize to the case appearing in the work of Zwegers, but we prefer to work generally for now to emphasize that the phenomena are not unique to the context in \cite{Zw}. Fix the following data; \begin{itemize}\item$\Gamma$ a discrete group acting on a Stein complex manifold, $X$,with finite stabilizers. \item An automorphy factor $j\in H^{1}(X,\mathcal{O}_{X}^{*})$, $j:\gamma\in \Gamma\mapsto j_{\gamma}\in \mathcal{O}_{X}^{*}$. \item A class $\delta\in H^{1}(X,\mathcal{O}(X)|^{j})$, where $\mathcal{O}(X)|^{j}$ denotes the $\Gamma$-module with $\gamma$ acting as $f\mapsto f|^{j}_{\gamma}$.\end{itemize}

By descent, this corresponds to the following geometric data; 
\begin{itemize} \item The analytic Deligne-Mumford stack $X//\Gamma$. \item A line bundle on $X//\Gamma$, which we will denote $\mathcal{L}_{j}$. \item A rank $2$ vector bundle, $\mathcal{V}_{\delta}$, fitting into an extension, $\mathcal{L}_{j}\rightarrow\mathcal{V}_{\delta}\rightarrow\mathcal{O}$, corresponding to the class $\delta\in H^{1}(X//\Gamma,\mathcal{L}_{j})$. \end{itemize}

\begin{lemma} A meromorphic trivialization of this bundle is equivalent to the data of a meromorphic function $s$, on $X$, satisfying the transformation property $$\gamma^{*}s=j_{\gamma}s+\delta_{\gamma},\, \gamma\in\Gamma.$$ \end{lemma}

\begin{proof} This is basically trivial, the point is that the bundle $\mathcal{V}_{\delta}$ is represented by the non-abelian cohomology class in $H^{1}(\Gamma,GL_{2}(\mathcal{O}_{X}))$, given by the matrices $$J_{\gamma}:=\begin{bmatrix}\,j_{\gamma}&\delta_{\gamma}\\ \,0&1\end{bmatrix}.$$ A section is then given by a vector $\binom{s_{0}}{s_{1}}$ transforming according to the automorphy factor $J$. A trivialization is given by a lift of the section $1$, thus a section of the form $\binom{s}{1}$, whence we are done.\end{proof}
Let us now say something about Dolbeault representatives for the relevant cohomology classes. We have a class $\delta\in H^{1}(X//\Gamma,\mathcal{L}_{j})$. A Dolbeault representative for this class is given by a smooth $(0,1)$-form on $X$ with the property $\gamma^{*}\lambda=j_{\gamma}\lambda$, modulo the $\bar{\partial}$-operator (acting on functions with the same transformation property). 

\begin{lemma} We recover the above cocycle representation of the class $\delta$ from $\lambda$ as follows - take any $\bar{\partial}$-antiderivative for $\lambda$, denoted $F$. Then the cocycle is given by $\gamma\mapsto F-F|^{j}_{\gamma}$.\end{lemma}

\begin{proof} First note that since we do not demand that $F$ has the same transformation properties as $\lambda$, such an $F$ necessarily exists and the space of such is a torsor for the space of holomorphic functions on $X$. Further note that whilst $F$ is not necessarily holomorphic the corresponding cocycle is, because of the equations $\bar{\partial}F=\lambda$ and $\gamma^{*}\lambda=j_{\gamma}\lambda$. It is not hard to show that this is the desired cocycle. \end{proof}

\begin{corollary} If $\binom{s}{1}$ is a section of $\mathcal{V}_{\delta}$ lifting $1\in H^{0}(\mathcal{O})$, then the $\mathcal{C}^{\infty}$ function, $s+F,$ transforms according to the automorphy factor $j$. \end{corollary}
As such, we see that $F$ is simultaneously a trivialization of the extension class, expressed in either the language of automorphy factors or of $\bar{\partial}$-cohomology. Presumably, though we will not claim to understand the physics, this is an instance of the principle one sees in the physics literature, whereby \emph{modular anomalies} and \emph{holomorphic anomalies} are essentially interchangeable. We will refer to $F$ as \emph{controlling} the extension.

\begin{example}Let us describe a toy model. We can take $X=\mathbb{C}$, $\Gamma=\mathbb{Z}^{2}$ so that the quotient is the elliptic curve $E_{\tau}$ with coordinate $z$. We take the automorphy factor to be trivial. We have a class $\delta\in H^{1}(E,\mathcal{O})$, given in automorphic terms as the element of $\delta\in H^{1}(\mathbb{Z}^{2},\mathcal{O}(\mathbb{C}))$ with $\delta\binom{1}{0}=0, \,\,\delta\binom{0}{1}=1.$ A Dolbeault representative is given by the $(0,1)$-form $\frac{1}{2iy}d\bar{z}$, with $y:=\mathfrak{Im}\tau$. The relevant non-holomorphic function $F(z,\bar{z})$ controlling the extension can be taken to be $$F(z,\bar{z}):=\frac{\bar{z}-z}{2iy}.$$ A meromorphic trivialization of the extension corresponding to $\delta\in H^{1}$ is given by the Weierstrass $\zeta$ function (a suitably normalized indefinite integral of $\wp(z)$) and we have that $\zeta(z)+F(z,\bar{z})$ is $\mathbb{Z}^{2}$-invariant. \end{example}

\subsection{Dictionary With Zwegers}We now sketch how to specialize the above considerations to the context studied in the work of Zwegers, nothing is susbtantially new here, we just supply some algebro-geometric intepretations. Recall that we set $x=e^{2\pi iu}, y=e^{2\pi iv}, q=e^{2\pi i\tau}$, with $\mathfrak{Im}(\tau)>0$. In the notation of the previous subsection \begin{definition}We take $X:=\mathbb{C}\times\mathbb{C}\times\mathfrak{h}$, $\Gamma:=SL_{2}\rtimes(\mathbb{Z}^{2}\times\mathbb{Z}^{2})$. The modular parameter will be denoted $\tau$ and the two (additive) elliptic parameters $u$ and $v$. We refer to the analytic stack $X//\Gamma$ as $\mathcal{C}$. \end{definition}

\begin{remark} Putting $\mathcal{M}:=(\mathbb{C}\times\mathfrak{h})//SL_{2}\rtimes\mathbb{Z}$, it is standard that $\mathcal{C}\rightarrow\mathcal{M}$ is the universal curve on the moduli stack of pairs consisting of an elliptic curve and a degree $1$ line bundle on it. A pair $(v;\tau)\in \mathcal{M}(\mathbb{C})$ determines an elliptic curve $E_{\tau}$, and the line bundle $\Theta(u+v)\cong\Theta(u)\otimes L_{-v}$ on it.\end{remark}

As mentioned, $\mathcal{C}\rightarrow\mathcal{M}$ is a universal family, and so there is a line bundle on it realizing the various $\Theta(u+v)$ fibrewise, which is unique once we fix a normalization for the push-forward line bundle, ie what will become the odd framing. The resulting line bundle will be denoted $\Theta(u+v;\tau)$. It is determined by demanding that $\vartheta(u+v,\tau)$ is a section of it. Note that this is well defined as the zeroes of $\vartheta$ are invariant under the \emph{full} group $\Gamma:=SL_{2}\rtimes(\mathbb{Z}^{2}\times\mathbb{Z}^{2})$, because we are working with the $\vartheta$-function corresponding to the odd spin structure. A representative for the line bundle in automorphic terms is given by a cocycle $j$, so that for $\gamma\in SL_{2}(\mathbb{Z})$, we have $$j_{\gamma}(u,v;\tau)\sim \frac{1}{(c\tau+d)^{\frac{1}{2}}}\exp\Big( -\pi i\frac{(u+v)^{2}}{c\tau+d}\,\Big),$$ with constant of proportionality some eighth root of unity.
\begin{lemma} There is a rank $2$ vector bundle on the stack $\mathcal{C}$, denoted again $\mathcal{K}$, fitting into an extension $$\mathcal{O}_{\mathcal{C}}\rightarrow\mathcal{K}\rightarrow\Theta(u+v;\tau),$$ realizing fibrewise (over the moduli stack of elliptic curves) the extension of lemma 2.4. Further $\kappa$ gives a meromorphic splitting of the extension.\end{lemma}\begin{proof} Again we construct the dual bundle $\mathcal{K}^{*}$. It suffices to specify an appropriate automorphy factor in $H^{1}(\Gamma,GL_{2}(\mathcal{O}_{X}))$. Now the condition that $\mathcal{K}$ fits into an exact sequence as desired means we are looking for upper triangular automorphy factors and that the diagonal entries are specified as those corresponding to $\mathcal{O}$ and $\Theta(u+v,\tau)^{-1}$, whence we are only looking for the upper right entries. Equivalently we are looking for a cocycle $\delta\in H^{1}(\Gamma,\mathcal{O}(X)|^{j})$. Further, we know the values on $\mathbb{Z}^{2}\times\mathbb{Z}^{2}$, as we have given them above in lemma 2.4. 

It remains thus to specify $\delta_{\gamma}$ for $\gamma\in SL_{2}$. We write $T:= \tau\mapsto\tau+1$ and $S:=\tau\mapsto -\frac{1}{\tau}$ and stipulate $\delta_{T}=0$, $\delta_{S}(u,v;\tau)=\frac{1}{2i}h(u+v,\tau)$, where $h$ is the \emph{Mordell Integral} considered by Zwegers in \cite{Zw}. To see that in this manner we obtain an automorphy factor note that the results of \cite{Zw} (cf. in particular proposition 1.5 (2) of \emph{loc. cit.}), imply that the vector of meromorphic functions $\binom{\kappa(u,v;\tau)}{1}$ transforms according to the putative automorphy factor, whence we are done as we need only check the cocycle condition for the off-diagonal entries of the automorphy factors. That $\kappa$ gives a meromorphic splitting is now clear.\end{proof}

\begin{remark} Recall that Zwegers defines a non-holomorphic function $R(u;\tau)$ with the property that his completed function $\tilde{\mu}:=\mu+\frac{i}{2}R(u-v;\tau)$ is a section of $\Theta(u-v;\tau)^{-1}$. This of course translates trivially into such a function, which we denote $F$, so that we have $\kappa+F$ a section of $\Theta(u+v;\tau)$, let us now see that the results of Zwegers are in line with the generalities of subsection 3.1. Indeed we can easily check that lemma 1.8 of \emph{loc. cit.} provides precisely a Dolbeault representative of the cohomology class of the extension, further it is almost tautological that $R$ defines a $\mathcal{C}^{\infty}$ trivialization of the extension class expressed in terms of automorphy factors. We summarize this as follows; \begin{itemize}\item The Mordell Integral (cf. \cite{Zw} subsection 1.2) $h$ arises as a matrix coefficient of an automorphy factor for a rank two bundle on $\mathcal{C}$, which is an extension of $\mathcal{O}$ by $\Theta(u+v;\tau)^{-1}$. \item The extension class in $H^{1}(\mathcal{C},\Theta(u+v,\tau)^{-1})$ has a Dolbeault representative $[\bar{\partial}R]$, cf lemma 1.8 of \emph{loc. cit.}. As such, $R$, controls the extension in the sense of 3.2 above.\item $\mu$ is a meromorphic trivialization of the extension, so that $\mu +R$ is necessarily a smooth section of the sub-bundle $\Theta(u+v;\tau)^{-1}$.\end{itemize}\end{remark}

We let $\mathcal{M}^{\mathcal{N}=2}_{ell}$ denote the moduli of $\mathcal{N}=2$ elliptic curves with odd fibre direction a degree one line bundle, equipped with an odd framing. Let $\mathcal{C}^{\mathcal{N}=2}_{ell}\rightarrow\mathcal{M}^{\mathcal{N}=2}_{ell}$ denote the universal family of such $\mathcal{N}=2$ curves. 
Note that the reduced parts of these spaces are respectively just $\mathcal{M}$ and $\mathcal{C}$. Further, the universal family over the reduced part of the base is $(\mathcal{C}\,|\,\Theta(u+v,\tau))\rightarrow\mathcal{M}$. The $0|1$-dimensional thickening providing the genuine moduli space is of course determined by the bundle $\mathcal{K}$. \begin{tcolorbox}\begin{theorem} There are isomorphisms, $\mathcal{C}^{\mathcal{N}=2}_{ell}\cong(\mathcal{C}\,|\,\mathcal{K})$ and $(\mathcal{M}\,|\,\mathcal{O})\cong\mathcal{M}^{\mathcal{N}=2}_{ell}.$ The universal family is induced from the short exact sequence $\mathcal{O}\rightarrow\mathcal{K}\rightarrow\Theta(u+v;\tau).$ It is not a split family, and has a meromorphic splitting given by the Appell-Lerch sum $\kappa(u,v;\tau)$. There is an $\mathbb{R}$-analytic splitting given by Zwegers' correction function $R(u+v;\tau)$. \end{theorem}\end{tcolorbox}

\begin{proof} The proof is essentially identical to the proof of corollary 2.1 above. \end{proof}

\begin{remark} Note that $R$ is real-analytic. The existence of a \emph{smooth} such $R$ is a formal consequence of the existence of partitions of unity, it is considerably less obvious fact that a real-analytic such also exists by general cohomological arguments. This follows from Proposition 2.3 of \cite{AH}, which says that real-analytic coherent cohomology vanishes.\end{remark}

\end{document}